\documentclass[a4paper,11pt,reqno]{amsart}
\usepackage[utf8]{inputenc}
\usepackage{amsmath}
\usepackage{amsfonts}
\usepackage{amssymb}
\usepackage{graphicx}
\usepackage{times}
\usepackage{color}
\usepackage{tikz}
\usepackage{tikz-3dplot}
\usepackage{enumerate}
\usepackage{url}
\usepackage{ifthen}

\usepackage{fancybox}
\usepackage[framemethod=TikZ]{mdframed} 
\usepackage[left=2cm,right=2cm,top=2.0cm,bottom=2cm]{geometry}
\usepackage{eso-pic}
\usepackage{multirow}
\usepackage{multicol}

\usepackage{indentfirst}
\usepackage[T1]{fontenc}

\newtheorem{theorem}{Theorem}[section]
\newtheorem{definition}[theorem]{Definition}

\newtheorem{corollary}[theorem]{Corollary}
\newtheorem{lemma}[theorem]{Lemma}

\def\rad{\mathop{\rm rad}}
\def\ord{\mathop{\rm ord}}
\def\F{\mathbb F}
\def\R{\mathbb R}

\def\N{\mathbb N}
\def\Q{\mathbb Q}

\numberwithin{equation}{section}

\author[F.E. Brochero Mart\'{\i}nez]{F. E. Brochero Mart\'{\i}nez}
\author[Lays Silva-Jesus]{Lays Silva-Jesus} 
\address{
Departamento de Matem\'{a}tica\\
Universidade Federal de Minas Gerais\\
UFMG\\
Belo Horizonte, MG\\
 30123-970\\
 Brazil\\
 }
 \email{fbrocher@mat.ufmg.br }\email{laysgrazielle@gmail.com}

\title{The estimation of the number of Irreducible Binomials}

\date{\today}

\subjclass[2000]{ }


\begin{document}
\maketitle
\begin{abstract}
Let $\mathbb{F}_q$ be the finite field with $q$ elements, and $T$ a positive  integer.  In  this article we find a sharp estimative of  the total number of monic irreducible binomials  in $\F_q[x]$ of degree less or equal to $T$,  when $T$ is large enough.   
\end{abstract}

\section{Introduction}
Let $\F_q$ be a finite field with $q$ elements. 
Determining an asymptotic formula for the number of monic irreducible polynomial in $\F_q[x]$ of degree $t$ that satisfied some condition is an interesting question.  For instance,  in \cite{Coh} and  their references  can found a several results about asymptotic formulas of the number of  monic irreducible polynomials with some fixed coefficients,  and  in  \cite{Sti} found an asymptotic  formula for the number of irreducible polynomial invariant by some action of $PSL(2,\F_q)$.
In this direction, a natural question is finding a asymptotic function of the number of irreducible polynomial with a few nonzero coefficients, and in particular, when the polynomials  have only two nonzero coefficients.  
We point out that in that case, it is known an easy test to determine if that kind of polynomial is irreducible. In fact, 
the criteria is given by the following lemma.
\begin{lemma}\label{binomios} \cite[Theorem~3.75]{LiNi} Let $t\ge  2$ be an integer and $a \in \F_q^*$ . Then the binomial $x^t - a$ is irreducible in $\F_q [x]$ if and only if the following three conditions are satisfied:
\begin{enumerate}
\item Every prime divisor $p$ of $t$  divides $\ord_q a$.
\item $\gcd(t, \frac {q-1}{\ord_q a})=1$
\item  If  $4$ divides $t$  then $q \equiv 1 \pmod 4$.
\end{enumerate}
\end{lemma}
Thus, a natural question  is, by fixing $q$, what is the average of the number of binomials? i.e. 
If $t$ is positive integer and $N_q(t)$ denotes the number of $a\in \F_q$ such that the binomial $x^t-a$ is irreducible over $\F_q[x]$,  then, what is the behavior  of the sum $\sum_{t\le T} N_q(t)$  for $T\in \N$ large enough?
 
From Lemma \ref{binomios}, it is easy to determine  the function $N_q(t)$, specifically:
\begin{lemma}[ {\cite[Corollary 1 (b)]{BrGiOl}}  or {\cite[Lemma 7]{HeSh}}]\label{Br} Let $\F_q$ be a field with $q$ elements and $N_q(t)$ the number of monic irreducible binomials of degree $t$ in $\F_q$. Then
$$N_q(t)= \begin{cases}
\frac {\varphi(t)}t (q-1), & \text{if } \rad_4(t)|(q-1)\\
0, &\text{otherwise}
\end{cases}
$$
where $rad(t)$ is the product of the primes $p$ that divides $t$, $rad_4(t)=\begin{cases} \rad(t) & \text{if $4\nmid t$}\\ 2\rad (t) &\text{otherwise} 
\end{cases}$ and  $\varphi$ is the Euler Totient function.
\end{lemma}
 
The sum $\sum_{t\le T} N_q(t)$  was studied  in \cite{HeSh} by  Heyman and Shparlinski, using advance results of analytic number theory. They  estimated the total number of irreducible binomials on average over $q$ or $t$. In particular they calculated a higher bound of the  average the number of irreducible binomial when $q$ is fixed and $t\le T$, where $T$ is large enough. One result of that article is the following:
 \begin{theorem} \cite[Theorem~3.]{HeSh}
For any fixed positive $A$ and $\epsilon$ and a sufficiently large real $q$ and $T$ with
$$T\ge  (\log(q - 1))^{ (1+\epsilon)A \log_3 q/ \log_4 q}$$
we have
$$\sum_{t\le T} N_ q (t) \le  (q - 1)\frac {T}{ (\log T )^ A} .$$
\end{theorem}
We will see that this upper bound is not close to that the  actual value.  In fact, in this article, using elementary tools, we give the exact order of growth for the average order of the number of irreducible binomials of degree $ t $ on the field $\mathbb {F}_ q$, that which improves the result found by  Heyman and Shparlinski.  Specifically we prove the following result:

\begin{theorem}\label{principal}
Let $\F_q$ be a finite fields with $q$ elements and $N_q(t)$ the number of monic irreducible binomials of degree  $t$ in $\F_q[x]$. 
\begin{enumerate}
\item 
If $q\not\equiv 3 \pmod 4$ and $q-1=p_1^{\alpha_1}\dots p_s^{\alpha_s}$ is the factorization of $q-1$ in prime factors,  then 
$$\sum_{1\le t\le T} N_q(t)= \frac{\varphi (q-1)}{s!\log p_1\cdots \log p_s}\left( (\log T)^s+\frac s2 \sum_{j=1}^s \frac{(p_j+1)\log p_j}{p_j-1} (\log T)^{s-1}\right) +o\left(\frac{(\log T)^{s-1}}{\log\log T}\right).$$
\item If $q\equiv 3 \pmod 4$ and $q-1=2p_1^{\alpha_1}\dots p_s^{\alpha_s}$
,  then 
$$\sum_{1\le t\le T} N_q(t)= \frac{3\varphi (q-1)}{2s!\log p_1\cdots \log p_s}\left( (\log T)^s+\frac s2\left[ \sum_{j=1}^s \frac{(p_j+1)\log p_j}{p_j-1} -\frac {\log 4}3\right](\log T)^{s-1}\right) +o\left(\frac{(\log T)^{s-1}}{\log\log T}\right).$$
\end{enumerate}

\end{theorem}

In particular,  as a direct consequence we get the following corollary.

 \begin{corollary}
Let $q$ be a power of a prime such that $q-1=\begin{cases} p_1^{\alpha_1}\cdots p_s^{\alpha_s} &\text{if $q\not\equiv 3\pmod 4$}\\
2p_1^{\alpha_1}\cdots p_s^{\alpha_s} &\text{if $q\equiv 3\pmod 4$}
\end{cases}
$. Then
$$ \frac {s!\log p_1\cdots \log p_s}{\varphi(q-1)}\cdot \lim_{T\to \infty}\frac {\sum_{1\le t\le T} N_q(t)}{(\log T)^s}=\begin{cases} 1 &\text{if $q\not\equiv 3\pmod 4$}\\
\frac 32 &\text{otherwise.}
\end{cases}$$
\end{corollary}

\section{Notation and some useful result about lattice points}

Throughout this paper, $\N_0$ and $\N$ denote respectively  the set of non negative integer and the set of positive integers, $\F_q$ denotes a finite field with $q$ elements and  $q-1=\begin{cases}p_1^{\alpha_1}\cdots p_s^{\alpha_s}&\text {if $q\not\equiv 3\pmod 4$}\\
2p_1^{\alpha_1}\cdots p_s^{\alpha_s}&\text {otherwise,}
\end{cases} $
where $p_1,\dots,p_s$ are different primes.  For each positive integer $n$, $\varphi(n)$ denotes the Euler totient functions evaluated at $n$,  $N_q(n)$ denotes the number of monic irreducible binomial in $\F_q[x]$, $\rad(n)$ is the product of prime divisor of $n$ and $\rad_4(n)=\begin{cases} \rad(n) &\text{if $4\nmid n$}\\ 2\rad(n) &\text{otherwise}\end{cases}$ . 
For each $\vec v=(v_1,\dots,v_s)\in (\N_0)^s$, we denote by $t(\vec v)$ the number $p_1^{v_1}\cdots p_s^{v_s}$.   Finally, $l_j$ denotes the number $\log p_j$ for $j=1,\dots, s$. 

In what follows, we present some result about the number of lattice points in the $s$-dimensional tetrahedron  bound by the coordinates hyperplanes and an appropriate plane.    These estimate  will be  use to determine how many positive integers $t$ satisfied $t\le T$ and $\rad(t)|(q-1)$. Specifically,  we want to know how many integer numbers of the form $p_1^{v_1}\cdots p_s^{v_s}$ are less or equal to $T$.   For this purpose, we need some definitions and results, which we will present below.

\begin{definition} Let $a_1,\dots, a_s$ and $\lambda$ be positive real numbers. Let define
$\Omega(\lambda; a_1,\dots, a_s)$ as the close tetrahedron  limited by the coordinates hyperplanes and the hyperplane   $a_1x_1+\cdots +a_sx_s= \lambda$, i.e. 
$$\Omega(\lambda; a_1,\dots, a_s)=\{ (x_1,\dots, x_s)\in \R^s|    x_i\ge 0 \text{ and } a_1x_1+\cdots +a_sx_s\le \lambda\},$$
and $\mathcal N(\lambda; a_1,\dots, a_s)$ the number of points of  the set $\N_0^s\cap \Omega(\lambda; a_1,\dots, a_s)$, i.e.
$$\mathcal N_s(\lambda; a_1,\dots, a_s)=|\{ (x_1,\dots, x_s)\in \N_0^s|    a_1x_1+\cdots +a_sx_s\le \lambda\}|.$$
\end{definition}

A naive estimation  of the number of lattice points in the tetrahedron is a classical result:  If we consider, for each element of $\Omega(\lambda; a_1,\dots, a_s)$ with integer coordinates, a hypercube of size 1 located in the positive direction respect that point, then the solid $\mathcal C_s$ obtained from the union these hypercubes contains $\Omega(\lambda; a_1,\dots, a_s)$.  Thus 
\begin{equation}\label{trivialdown}
\mathcal N_s(\lambda; a_1,\dots, a_s)=Vol(\mathcal C_s)>Vol(\Omega(\lambda; a_1,\dots, a_s))=  \frac 1{s!} \prod_{j=1}^s \frac {\lambda}{a_j}=\frac {\lambda^s}{s! a_1\cdots a_s}.
\end{equation}
The same way,  $\mathcal C_s$ is contained in  $\Omega(\lambda+a_1+\cdots+a_s; a_1,\dots, a_s)$, so
\begin{equation}\label{trivialup}
\mathcal N_s(\lambda; a_1,\dots, a_s)<Vol(\Omega(\lambda+a_1+\cdots+a_s; a_1,\dots, a_s))= \frac {(\lambda+a_1+\cdots+a_s)^s}{s! a_1\cdots a_s}.
\end{equation}
Therefore,  the function $\mathcal N_s$ can be bound lower and upper by two polynomials of degree $s$ in the variable $\lambda$  for all $\lambda>0$ and asymptotically we have   $\mathcal N_s(\lambda; a_1,\dots, a_s)= \frac {\lambda^s}{s! a_1\cdots a_s}+ O( \lambda^{s-1})$.
We note that, in \cite{Leh} Lehmer determines  two other polynomials $P_{a_1,\dots,a_s}(\lambda)$  and $Q_{a_1,\dots,a_s}(\lambda)$ such that  bound more efficiently the function $\mathcal N_s$, 
 for all $\lambda>0$. Analogous results were found by  Lochs \cite{ Loch2},  that we can summarize, in a simplified form, in the following theorem 
\begin{theorem}[{\cite{Leh}} and {\cite{Loch2}}]\label{Leh_theo} Let $a_1, \dots,a_s$ be a real number.  Then
\begin{equation}\frac{ {\lambda^s}+\frac s2(a_2+\cdots+a_s)\lambda^{s-1}}{s! a_1\cdots a_s}<\mathcal N_s(\lambda; a_1,\dots, a_s)<\frac{( \lambda+\frac 12(a_1+\cdots +a_s))^s}{s! a_1\cdots a_s}, \text{ for all $\lambda>0$}.\end{equation}
\end{theorem}
On other hand, in \cite{Spen} Spencer  found a asymptotic formula for the function $\mathcal N_s$.
\begin{theorem}[{\cite[Theorem I]{Spen}}]\label{Spencer} Let $a_1, \dots,a_s$ be  real numbers, that are linearly independent over $\Q$.  Then
$$\mathcal N_s(x; a_1,\dots,a_s)= \frac {x^s}{s! a_1\cdots a_s} + \frac 1{2 (s-1)!} \frac {a_1+\cdots +a_s}{a_1\cdots a_s} x^{s-1}+ o(x^{s-1}/\log x).$$
\end{theorem}
An elementary proof of this  result can be found in  \cite{Beu}.

\section{Irreducible Binomials and Associate Lattice}
 
For each  positive integer $T$, let  denote by $\Upsilon(T)$ the set of lattice points 
$$\Omega(\log T;l_1,\dots, l_s)\cap \N_0^s, \text{ where $l_j=\log p_j$ for $j=1,\dots,s$}.$$ 
It is clear, from the definition of $\Upsilon(T)$  that 
$$\vec v:=(v_1,\dots, v_s)\in \Upsilon(T)\text{ if and only if } t(\vec v)=p_1^{v_1}\cdots p_s^{v_s}\le T.$$
We also denote by $\Upsilon^+(T)=\Upsilon(T)\cap \N^s$, i.e.  the elements of $\Upsilon(T)$ with every coordinates are positive, and for each $j=1,\dots,s$,  $\Upsilon_j(T)$  denotes the subset of $\Upsilon(T)$ with the $j$-th coordinates is equal to zero and the other coordinates are positive. By definition, the sets 
$$\Upsilon^+(T), \Upsilon_1(T),\dots, \Upsilon_s(T)$$
 are disjoint in pairs.  Finally we denote by $\Upsilon_0(T)$ the complementary subset
 $$\Upsilon(T)\setminus (\Upsilon^+(T)\cup\Upsilon_1(T)\cup\dots\cup \Upsilon_s(T)),$$
i.e. the subset of elements with two or more coordinates equal to zero.
 
In order to proof the main result of this article, we need the following lemma, that will be use to  estimate  the number of irreducible binomial   of degree $t\le T$, where $\frac {\rad(q-1)}{\rad(t)}$ is $1$ or a prime number.   The essential idea  is that the number of binomials of this type  is asymptotically bigger that the other type of binomials of degree less that or equal to $T$.

 \begin{lemma}\label{lem3.1}Let $T> \rad(q-1)$ be an integer. 
 Then
 \begin{enumerate}[\bf(a)]
 \item  
 \begin{align*}\sum_{\vec v\in \Upsilon^+(T)} \frac {\varphi (t(\vec v))}{t(\vec v)}&=\frac{\varphi(q-1)}{q-1} \mathcal N_s\left(\log \left(\frac T{\rad(q-1)}\right); l_1,\dots, l_s\right)\\
 &= \frac{\varphi(q-1)}{(q-1)s! l_1\cdots l_s} (\log T)^s\left(1-  \frac{s \log(\rad(q-1))}{2\log T}\right)+o\left(\frac{(\log T)^{s-1}}{\log\log T}\right)
 \end{align*}
  
 \item  
  \begin{align*}\sum_{j=1}^s\sum_{\vec v\in \Upsilon_j(T)} \frac {\varphi (t(\vec v))}{t(\vec v)}&=\frac{\varphi(q-1)}{q-1}\sum_{j=1}^s \frac {p_j}{p_j-1} \mathcal N_{s-1}\left(\log \left(\frac {p_jT}{\rad(q-1)}\right); l_1,\dots,
\widehat {l_j},
\cdots,l_s\right)  \\
  &=\frac{\varphi(q-1)}{(q-1)(s-1)! l_1\cdots l_s} (\log T)^{s-1}\sum_{j=1}^s \frac {p_j l_j}{p_j-1}  +O( (\log T)^{s-2})
  \end{align*}
where $\widehat {l_j}$ means that $l_j$ does not appear as parameter in the function.
 \end{enumerate}
 \end{lemma}  
 \begin{proof}
 {\bf(a)}  Since $\rad(t(\vec v))=\rad(q-1)$ for all $\vec v \in \Upsilon^+(T)$,  it follows that 
$$\frac{\varphi (t(\vec v))}{t(\vec v)}= \frac{\varphi (\rad(t(\vec v)))}{\rad(t(\vec v))}=\frac{\varphi(\rad(q-1))}{\rad(q-1)}=\frac {\varphi(q-1)}{q-1}.$$
Therefore
\begin{equation}\label{eq3.1}
\sum_{\vec v\in \Upsilon^+(T)} \frac {\varphi (t(\vec v))}{t(\vec v)}=\frac {\varphi(q-1)}{q-1} |\Upsilon^+(T)|.
\end{equation}
On other hand, we know that
$$(v_1,v_2,\dots,v_s)\in \Upsilon^+(T) \text{ if and only if } (v_1-1,v_2-1,\dots, v_s-1)\in \Upsilon\left(\frac {T}{p_1\cdots p_s}\right),$$
thus 
\begin{equation}\label{upsilon1} |\Upsilon^+(T)|= \left|\Upsilon\left(\frac {T}{p_1\cdots p_s}\right)\right|=\mathcal N_s\left(\frac{T}{p_1\cdots p_s}; l_1,\dots,l_s\right).
\end{equation}
Finally, from Theorem \ref{Spencer} and using the fact that 
$$ \left(\log\left(\frac T{p_1\cdots p_s}\right)\right)^{k}=(\log T)^k-k\log(p_1\cdots p_s)(\log T)^{k-1} +O((\log T)^{k-2})$$ for all $k\ge 1$,  we have that
\begin{align*}
 |\Upsilon^+(T)|&=
\frac 1{s! l_1\cdots l_s} \left( \left(\log\left(\frac T{p_1\cdots p_s}\right)\right)^s+\frac s2(l_1+\cdots+l_s)  \left(\log\left(\frac T{p_1\cdots p_s}\right)\right)^{s-1}   \right)+ o\left( \frac{(\log T)^{s-1}}{\log\log T}\right)\\
&=
\frac 1{s! l_1\cdots l_s} \left( 
(\log T)^s-s\log(p_1\cdots p_s)(\log T)^{s-1}+\frac s2\log(p_1\cdots p_s) (\log T)^{s-1}\right) +o\left( \frac{(\log T)^{s-1}}{\log\log T}\right)\\
 &= \frac{(\log T)^s}{s! l_1\cdots l_s} \left(1-  \frac{s \log(\rad(q-1))}{2\log T}\right)+o\left(\frac{(\log T)^{s-1}}{\log\log T}\right).
\end{align*}
The result follows from Equation \ref{eq3.1} and this last identity.

{\bf (b)} If $\vec v\in \Upsilon_j(T)$, then $\rad(t(\vec v))=\frac {\rad(q-1)}{p_j}$ and 
$$\frac{\varphi (t(\vec v))}{t(\vec v)}=\frac{\varphi\left(\frac{\rad(q-1)}{p_j}\right)}{\frac{\rad(q-1)}{p_j}}=\frac {\varphi(\rad(q-1))}{\rad(q-1)}\cdot\frac {p_j}{p_j-1}=\frac {\varphi(q-1)}{q-1}\cdot\frac {p_j}{p_j-1}.$$
Therefore
\begin{equation}\label{eq3.2}\sum_{j=1}^s\sum_{\vec v\in \Upsilon_j(T)} \frac {\varphi (t(\vec v))}{t(\vec v)}=\frac{\varphi(q-1)}{q-1}\sum_{j=1}^s \frac {p_j}{p_j-1}|\Upsilon_j(T)|.
\end{equation}
Using the fact that every point in $\Upsilon_j(T)$ has the $j$-th coordinate equal to zero, we can eliminate that coordinate and following the same reasoning as  before, we have that 	
\begin{align} 
|\Upsilon_j(T)|&=\mathcal N_{s-1}\left(\frac{T}{p_1\cdots p_{j-1}p_{j+1}\cdots p_s}; l_1,\dots,\widehat{l_{j}},\dots, l_s\right)\nonumber\\
&=\frac 1{(s-1)! l_1\cdots l_{j-1}l_{j+1}\cdots l_s}(\log T)^{s-1}+ O((\log T)^{s-2}).\label{eq3.3}
\end{align} 
The result  follows from Equations \ref{eq3.2}  and \ref{eq3.3}.
 
 \end{proof}
 \section{Proof of Theorem \ref{principal}}
Firstly, let us consider the case when $q\not \equiv 3 \pmod 4$, and then   $\rad_4(q-1)=\rad(q-1)$. Since $N_q(t)=0$ if $\rad_4(t)\nmid (q-1)$, we have from Lemma \ref{Br} that
\begin{equation}
\sum_{t\le T}N_q(t)=\sum_{{t\le T\atop \rad(t)|(q-1)}} N_q(t)=(q-1)\!\sum_{{t\le T\atop \rad(t)|(q-1)}} \frac{\varphi(t)}{t}
\end{equation}
Observe that the conditions $t\le T$ and $\rad(t)\mid(q-1)$ is equivalent to $t=p_1^{v_1}\cdots p_s^{v_s}\le T$ where each $v_j$ is  a non negative integer. This last inequality is equivalent to the linear inequality  $v_1\log p_1+\cdots +v_s\log p_s\le \log T$, i.e. $\vec v=(v_1,\dots,v_s)\in\Upsilon(T)= \Omega(\log T; l_1,\dots, l_s)\cap \N_0^s$
. Therefore
$$\sum_{t\le T}N_q(t)=
(q-1)\!\sum_{\vec v\in \Upsilon(T)}\frac{\varphi(t(\vec v))}{t(\vec v)}
= (q-1)( A+B+C),$$
where 
\begin{equation}\label{ABC}
A:= \sum_{\vec v\in \Upsilon^+(T)}\frac{\varphi(t(\vec v))}{t(\vec v)},\qquad 
B:=\sum_{j=1}^s\sum_{\vec v\in \Upsilon_j(T)}\frac{\varphi(t(\vec v))}{t(\vec v)}\quad
\text{ and }\quad C:=\sum_{\vec v\in \Upsilon_0(T)}\frac{\varphi(t(\vec v))}{t(\vec v)}. \end{equation}
The summations $A$ and $B$  correspond to items (a) and (b) of Lemma \ref{lem3.1}.  Thus
\begin{align*}
(q-1)(A+B)&=\frac{\varphi(q-1)}{s! l_1\cdots l_s}\left( (\log T)^s-  \frac{s}2\sum_{j=1}^s l_j(\log T)^{s-1}
+s \sum_{j=1}^s \frac {p_j l_j}{p_j-1}(\log T)^{s-1}
\right)+o\left(\frac{(\log T)^{s-1}}{\log\log T}\right)\\
&=\frac{\varphi(q-1)}{s! l_1\cdots l_s}\left( (\log T)^s+  \frac{s}2\sum_{j=1}^s\frac {p_j+1}{p_j-1} l_j(\log T)^{s-1}
\right)+o\left(\frac{(\log T)^{s-1}}{\log\log T}\right)
\end{align*}
The last summation can be bound as
\begin{equation}\label{faltamdois}
\sum_{\vec v\in \Upsilon_0(T)}\frac{\varphi(t(\vec v))}{t(\vec v)}\le \sum_{\vec v\in \Upsilon_0(T)} 1\le
\sum_{1\le i<j\le s}|\Upsilon_{ij}(T)|,
\end{equation}
where 
\begin{equation}\label{faltamdoisII}
\Upsilon_{ij}(T):=  \sum_{1\le i<j\le s} |\{\vec v\in \Upsilon_0(T); v_i=0 \text{ and } v_j=0\}|.
\end{equation}
Since $\Upsilon_{ij}(T)=O((\log T)^{s-2})$, we conclude that the last term is asymptotically small compared to the firsts two summations.

Finally, the proof of the case when $q\equiv 3\mod 4$ is essentially the same, using the fact that
$$\sum_{t\le T}N_q(t)=
\sum_{t\le T\atop \rad_4(t)|(q-1)}N_q(t)=
\sum_{t\le T\atop \rad(t)|\frac {q-1}2}N_q(t)+\sum_{t\le T/2\atop \rad(t)|\frac {q-1}2}N_q(2t),$$
and these two summations are similar to the before case.  
\qed

\section{Bound the number of irreducible binomials for  $T$ small}
The purpose of this section is to show some lower and upper bounds for the number of monic irreducible binomials where $T$ is not necessarily a big number.  It is clear that these bound can be improved for  more complicated functions, but for clarity we will no try to do that.  

Let suppose that $q\not\equiv 3\pmod 4$ and $s\ge 2$.  We can obtain a trivial lower and upper bound using Theorem \ref{Leh_theo} 
and observing that the inequalities $1\ge \frac{\varphi(t)}{t}\ge \frac{\varphi(q-1)}{q-1}$ are true for any $t$ such that $\rad(t)|(q-1)$. Thus,  for any $T>1$  we have that
$$\sum_{t\le T}N_q(t)= (q-1)\!\!\sum_{t\le T \atop \rad(t)|(q-1)} \frac{\varphi(t)}t\ge \varphi(q-1)|\Upsilon(T)|> \frac {\varphi(q-1)}{s! \log p_1\cdots \log p_s}(\log T)^s\left(1+ \frac {s\log(\rad(q-1)/p_1)}{2\log T}\right)$$
and 
$$\sum_{t\le T}N_q(t)= (q-1)\!\!\sum_{t\le T \atop \rad(t)|(q-1)} \frac{\varphi(t)}t\le (q-1)|\Upsilon(T)|< \frac {q-1}{s! \log p_1\cdots \log p_s}(\log T)^s\left(1+\frac{\log(\rad(q-1))}{2\log T}\right)^s.$$
In the following theorem, we improved the upper bound, which is very weak compared with the asymptotic result proven in the previous section. 

\begin{theorem}For any $T> \rad(q-1)$, where  $q\ne 3 \pmod 4$,
the number of monic irreducible binomial in $\F_q[x]$ of degree less or equal to $T$ is upper bounded by
$$\frac {\varphi (q-1)}{s!\log p_1\cdots \log p_s} (\log T)^s\left( 1+sM_1\frac {\log(\rad(q-1))}{\log T} + s(s-1)M_2\left(\frac {\log(\rad(q-1))}{\log T}\right)^2\right),$$
where 
$$M_1:=(\rad(q-1))^{-(s-1)/(2\log T)}\left(1+\frac {\log 2s}s\right)- \frac 12$$ 
and 
$$M_2:=
\frac{(q-1)(s-1)}{2s\varphi(q-1)} (\rad(q-1))^{(s-2)/(2\log T)} + \frac 18.$$

\end{theorem} 

\begin{proof}
We know that the number of monic irreducible binomials is given by the formula $(q-1)(A+B+C)$, where $A$, $B$ and $C$ are defined by  Equations in \ref{ABC}. Therefore it is enough to find an upper bound to each of them.  From Equation \ref{upsilon1} and Theorem \ref{Leh_theo}, we have that
\begin{align}
A&=\frac {\varphi(q-1)}{q-1}\left|\Upsilon\left(\log\left(\frac T{\rad(q-1)}\right)\right)\right|\nonumber\\
&\le \frac {\varphi(q-1)}{q-1}\cdot \frac {\left( \log\left(\frac T{\rad(q-1)}\right)+\frac 12(l_1+\cdots+l_s)\right)^s}{s! l_1\cdots l_s}\nonumber\\
&= \frac {\varphi(q-1)}{(q-1)s! l_1\cdots l_s}(\log T)^s\left( 1-\frac {\log (\rad(q-1))}{2\log T}\right)^s,\nonumber\\
&\le \frac {\varphi(q-1)}{(q-1)s! l_1\cdots l_s}(\log T)^s\left( 1-s\frac {\log (\rad(q-1))}{2\log T}+
s(s-1)\frac {(\log (\rad(q-1)))^2}{8(\log T)^2}
\right). \label{eqA}
\end{align}
From Equation \ref{eq3.2}, it follows that
\begin{align}
B&=\frac{\varphi(q-1)}{q-1}\sum_{j=1}^s \frac {p_j}{p_j-1}|\Upsilon_j(T)|\nonumber\\
&\le \frac{\varphi(q-1)}{q-1}\sum_{j=1}^s \frac {p_j}{p_j-1}\frac {\left(\log \left(\frac{p_j T}{\rad(q-1)}\right) +\frac 12(l_1+\cdots+l_{j-1}+l_{j+1}+\cdots+l_s)\right)^{s-1}}{(s-1)! l_1\cdots l_{j-1}l_{j+1}\cdots l_s}\nonumber\\
&= \frac{\varphi(q-1)}{(q-1)s!l_1\cdots l_s}(\log T)^{s-1}\sum_{j=1}^s  \frac {sp_j\log p_j}{p_j-1}\left(1 -\frac{\log\left(\frac{\rad(q-1)}{p_j}\right)}{2\log T}\right)^{s-1}\nonumber\\
&= \frac{\varphi(q-1)}{(q-1)s!l_1\cdots l_s}(\log T)^{s-1}\sum_{j=1}^s  \frac {sp_j\log p_j}{p_j-1}\left(1 -\frac{\log\left(\rad(q-1)\right)}{2\log T}\right)^{s-1}\nonumber\\
&< \frac{\varphi(q-1)}{(q-1)s!l_1\cdots l_s}(\log T)^{s-1}\sum_{j=1}^s  \frac {sp_j\log p_j}{p_j-1} (\rad(q-1))^{-(s-1)/(2\log T)}\nonumber \\
&< \frac{\varphi(q-1)}{(q-1)s!l_1\cdots l_s}(\log T)^{s-1}
\left( s+\log 2s\right)
\log(\rad(q-1)) (\rad(q-1))^{-(s-1)/(2\log T)},
\label{eqB}
\end{align}
where in the last two  inequalities we use respectively that $(1+x)<e^x$ and Chebyshev's sum inequality. 
Finally, from Equations \ref{faltamdois} and \ref{faltamdoisII}, and using the same argument as before, we obtain that
\begin{align}
C&\le \sum_{1\le i<j\le s}|\Upsilon_{ij}(T)|\nonumber\\
&\le \sum_{1\le i<j\le s} \mathcal N_{s-2} (\log T; l_1,\dots, \widehat{l_i},\dots, \widehat{l_j},\dots,l_s)\nonumber\\
&\le\frac 1{(s-2)!l_1\dots l_s} (\log T)^{s-2}\sum_{1\le i<j\le s} l_il_j 
\left(1+\frac {\log \left(\frac{\rad(q-1)}{p_ip_j}\right)}{2\log T}\right)^{s-2}\nonumber\\
&< \frac 1{(s-2)!l_1\dots l_s} (\log T)^{s-2}\cdot \sum_{1\le i<j\le s} l_il_j \cdot
\left(1+\frac {\log \left(\rad(q-1)\right)}{2\log T}\right)^{s-2}\nonumber\\
&<\frac 1{(s-2)!l_1\dots l_s} (\log T)^{s-2}\cdot \frac {s-1}{2s} (\log\rad(q-1))^2 \cdot (\rad(q-1))^{(s-2)/(2\log T)},
 \label{eqC}
\end{align}
where in the last inequality we use the inequality $\sum\limits_{1\le i<j\le s} x_ix_j\le \frac {s-1}{2s} \Bigl(\sum\limits_{1\le i\le s} x_i\Bigr)^2$. From the inequalities \ref{eqA}, \ref{eqB} and \ref{eqC} it follows the result. 
\end{proof}


\begin{thebibliography}{99}

\bibitem{Beu} Beukers, F.
{\it The lattice-points of n-dimensional tetrahedra.}
Indag. Math. {\bf 37} (1975), 365-372

\bibitem{Coh}  Cohen, S. D., {\it Irreductible polynomials–Prescribed coefficients}, in Handbook of finite fields, G.L. Mullen and D. Panario, eds., CRC Press, Boca Raton, 2013

\bibitem{HeSh} Heyman, R., Shparlinski I. E., {\it Counting irreducible binomials over finite fields}.  Finite Fields Appl. {\bf 38} 1-12 (2016).

\bibitem{Leh} Lehmer D. H., {\it The lattice points of an $n$-dimensional tetrahedron}. Duke J. Math. {\bf 7} 341-353 (1940).


\bibitem{Loch2} Lochs, G., {\it \"Uber die Anzahl der Gitterpunkte in einem Tetraeder. } 
Monatsh. Math. {\bf 56}, (1952). 233-239.



\bibitem{BrGiOl}  Brochero Mart\'inez, F.~E.,  Giraldo Vergara, C. R.,  de Oliveira, L., {\it Explicit factorization of $x^ n-1\in \mathbb F_q[x]$}. Des. Codes Cryptogr. {\bf 77} , no. 1, 277-286 (2015).


\bibitem{LiNi} Lidl, R., Niederreiter, H.,  {\it Introduction to finite fields and their applications}. Cambridge University Press New York, NY, USA 1986

\bibitem{Spen} Spencer, D. C. {\it The lattice points of tetrahedra.}
J. Math. Phys. Mass. Inst. Tech. {\bf21} (1942), 189-197

\bibitem{Sti}
H.~Stichtenoth and A.~Topuzo\u{g}lu.
{\it Factorization of a class of polynomials over finite fields.}
Finite Fields Appl. {\bf 18} (2012) 108--122.
\end{thebibliography}
\end{document}